\documentclass[12pt]{article}
 \usepackage[margin=1in]{geometry} 
\usepackage{amsmath,amsthm,amssymb,amsfonts}
 \usepackage{amssymb}

\usepackage[pagewise]{lineno}
\usepackage[utf8]{inputenc}

\usepackage{hyperref}
 \usepackage{csquotes}
\usepackage[utf8]{inputenc}
\usepackage[english]{babel}
\usepackage{xcolor}

\providecommand{\keywords}[1]
{
  \small	
  \textbf{\textit{Keywords:}} #1
}

\newcommand{\MSC}[1]{%
  \small
  \textbf{\textit{Mathematics Subject Classification:}} #1
}
\title{Global boundedness in the higher-dimensional fully parabolic chemotaxis with weak singular sensitivity and logistic source   }
\author{
    Minh Le\thanks{The Institute of Theoretical Sciences, Westlake University, China \texttt{(leminh@westlake.edu.cn)}} }
\date{}

\begin{document}
\maketitle

\begin{abstract}
We consider the following chemotaxis system under homogeneous Neumann boundary conditions in a smooth, open, bounded domain $\Omega \subset \mathbb{R}^n$ with $n \geq 3$:  
\begin{equation*} 
      \begin{cases}
          u_t = \Delta u - \chi \nabla \cdot \left( \frac{u}{v^k} \nabla v \right) + ru - \mu u^2,  & \text{in } \Omega \times (0,T_{\rm max}), \\
         v_t =  \Delta v - \alpha v + \beta u, & \text{in } \Omega \times (0,T_{\rm max}),  
      \end{cases}
\end{equation*}  
where $k \in (0,1)$, and $\chi, r, \mu, \alpha, \beta$ are positive parameters. In this paper, we demonstrate that for suitably smooth initial data, the problem admits a unique nonnegative classical solution that remains globally bounded in time when $\mu$ is sufficiently large.

\end{abstract}
\\
\MSC{35B35, 35K45, 35K55, 92C15, 92C17}

\numberwithin{equation}{section}
\newtheorem{theorem}{Theorem}[section]
\newtheorem{lemma}[theorem]{Lemma}
\newtheorem{remark}{Remark}[section]
\newtheorem{Prop}{Proposition}[section]
\newtheorem{Def}{Definition}[section]
\newtheorem{Corollary}{Corollary}[theorem]
\allowdisplaybreaks

\section{Introduction}
We are concerned with the following system of partial differential equations in a smooth open bounded domain $\Omega \subset \mathbb{R}^n$ with $n \geq 3$:
\begin{equation}  \label{1}
      \begin{cases}
          u_t = \Delta u - \chi \nabla \cdot \left( \frac{u}{v^k} \nabla v \right) + ru - \mu u^2 ,  \qquad &\text{in } \Omega \times (0,T_{\rm max}), \\
         v_t =  \Delta v - \alpha v + \beta u, \qquad &\text{in } \Omega \times (0,T_{\rm max}),  
      \end{cases}
\end{equation}  
where  $\chi>0$, $r >0$, $\mu>0$, $\alpha>0$, $\beta>0$, $k>0$, $T_{\rm max} \in (0, \infty]$ is the maximal existence time, and 
\begin{equation} \label{initial}
    \begin{cases}
        u_0 &\in C^{0} (\bar{\Omega}) \text{ is nonnegative with } \int_\Omega u_0 >0  \\
      v_0 &\in W^{1, \infty}(\Omega), \quad v_0> 0 \quad \text{in }\Omega.
    \end{cases}
\end{equation}
The system \eqref{1} is endowed with the homogeneous Neumann boundary conditions:
\begin{align} \label{bdry}
    \frac{\partial u}{\partial \nu}= 0, \qquad \frac{\partial v}{\partial \nu}=0, \qquad \text{on } \partial \Omega \times (0,T_{\rm max}).
\end{align}
Cells or microorganisms can secrete chemical signals and direct their movement in response to chemical stimuli. This phenomenon, referred to as chemotaxis, can be modeled using the system \eqref{1}. Here, $u(x,t)$ and $v(x,t)$ represent the cell density and chemical concentration at position $x$ and time $t$, respectively. The simplest version, where $k = r = \mu = 0$, is known as the Keller-Segel model, which was introduced in \cite{Keller2}. This model has been studied extensively from various perspectives, not only because of its practical applications but also due to its interesting mathematical properties. One well-known property is the critical mass phenomenon in two dimensions. This means that solutions exist globally and remain bounded in time when the initial population is below a certain threshold \cite{Dolbeault, Dolbeault1, NSY}, while solutions blow up if it exceeds that threshold \cite{Nagai1, Nagai2, Nagai3, Nagai4}. However, this phenomenon no longer holds in higher dimensions. In fact, it was proven in \cite{Winkler-2010} that there exist infinitely many initial conditions with arbitrarily small mass that lead to blow-up. In an attempt to prevent blow-up, a logistic source has been introduced and studied in \cite{Tello+Winkler} for the parabolic-elliptic case and in \cite{Winkler-logistic} for the fully parabolic case in any dimensional setting. These results assert that a sufficiently large $\mu$ can prevent the occurrence of both finite- and infinite-time blow-up.

The signal-dependent sensitivity $\frac{\chi}{v}$, arising from the Weber-Fechner law of stimulus perception, was first considered in the pioneering works \cite{Keller, Keller+Segel} in the context of chemotactic response. In \cite{Winkler+2011}, it was shown that solutions exist globally in time under the condition $\chi < \sqrt{\frac{2}{n}}$ for $n \geq 2$, though the question of whether these solutions remain globally bounded was left open. This was later resolved in \cite{Fujie}, where it was established that solutions are indeed globally bounded under the same constraint on $\chi$. This result naturally led to the question of whether $\sqrt{\frac{2}{n}}$ represents the optimal bound for $\chi$. A partial answer was provided in \cite{Lankeit_2016}, where it was demonstrated that for $n=2$, blow-up can still be prevented under the relaxed condition $\chi < 1.015$.

The interplay between singular sensitivity and the logistic source term, $ru - \mu u^2$, was first investigated in \cite{MKTAM}, where the authors established the global existence of solutions for fully parabolic case in two dimensions. A similar problem was studied in \cite{FWY} for the parabolic-elliptic case in dimension two, where the authors proved that solutions exist globally for any positive parameters $r$ and $\mu$. Furthermore, they demonstrated that solutions remain bounded under the conditions $r > \frac{\chi^2}{4}$ when $\chi \in (0,2)$ and $r > \chi - 1$ when $\chi \geq 2$. The same boundedness result was later obtained in \cite{Zhao+Zheng} for the fully parabolic case in two dimensions under the same restrictions on $r$ and $\chi$.  

For the weak singular sensitivity($k<1$) chemotaxis with logistic source, the parabolic-elliptic model was investigated in \cite{zh}, where the author established the global boundedness of solutions in two dimensions, provided that $\mu$ is sufficiently large. This largeness assumption was subsequently removed in \cite{Minh6}, where it was further shown that the sub-logistic damping term, $-\frac{\mu u^2}{\ln^\gamma(u+e)}$ with $\gamma \in (0,1)$, ensures global boundedness. Higher-dimensional cases were studied in \cite{Kurt}, where the author demonstrated that solutions exist globally for $k \in (0,1)$ and remain bounded for $k \in \left ( 0, \frac{1}{2}+\frac{2}{n} \right )$ with $n \geq 3$, provided that $\mu$ is sufficiently large. Very recently, the fully parabolic model in two dimensions was recently investigated in \cite{minh7}, where the author proved that blow-up is prevented by the quadratic logistic damping term, $-\mu u^2$, for any $\mu > 0$. The main objective of this paper is to extend this research to arbitrary dimensions. Specifically, we demonstrate that a sufficiently large $\mu$ is indeed sufficient to prevent blow-up in any dimension. Our main result is stated as follows:  

\begin{theorem}\label{thm1}
 Let $\Omega \subset \mathbb{R}^n$ with $n\geq 3$ be a smooth bounded domain and $k \in (0,1)$. Assume that \eqref{initial} holds then there exists $\mu_0=\mu_0(n,k,r,\chi, \beta)>0$ such that for any $\mu>\mu_0$, the system \eqref{1} under the boundary condition \eqref{bdry} possesses a unique classical solution $(u,v)$  such that  
\[
u, v \in C^0 \left(\bar{\Omega} \times [0,\infty)\right) \cap C^{2,1} \left(\bar{\Omega} \times (0,\infty)\right).
\]  
Furthermore, \( u \) and \( v \) are strictly positive and  uniformly bounded in \( \Omega \times (0, \infty) \) in the sense that 
\begin{align*}
     \sup_{t> 0} \left \{ \left \| u(\cdot,t) \right \|_{L^\infty(\Omega)} + \left \| v(\cdot,t) \right \|_{W^{1, \infty}(\Omega)} \right \} <\infty.
\end{align*}
\end{theorem}    
\begin{remark}
    The result is consistent with \cite{Kurt}, however, a precise quantitative estimate for $\mu_0$ is left to be open.
\end{remark}
The main difficulty in dealing with system \eqref{1} is that the logistic source disrupts the mass conservation of the solutions, which leads to the possibility that \( v \) can approach zero. This, in turn, causes the cross-diffusion term to become arbitrarily large, preventing it from being absorbed into the diffusion term. To overcome this obstacle, we introduce the energy function 
\[
y(t) = \int_\Omega u^p + \int_\Omega u^p v^{-q} + \int_\Omega v^{p+1},
\]
with \( p > 0 \) and \( q > 0 \), satisfying appropriate conditions. This energy functional, together with two key estimates—Lemma \ref{l1} and Lemma \ref{LK-1}—enables us to establish \( L^p \)-boundedness for the solutions, which further leads to their uniform boundedness. \\

The paper is organized as follows. In Section \ref{s2}, we establish the local existence of solutions and introduce several useful estimates that will be applied in the subsequent sections. In Section \ref{s3}, which is the core of our analysis, we derive key a priori estimates for the solutions. Finally, in Section \ref{s4}, we provide a brief proof of the main result.

\section{Preliminaries} \label{s2}
We begin this section by recalling a local existence result for system \eqref{1}, stated in the following lemma.
\begin{lemma} \label{local-exist}
    Let $\Omega \subset \mathbb{R}^n$ with $n \geq 3$ be a smooth, bounded domain, and let $\chi, k, r, \mu, \alpha, \beta$ be positive parameters. Assume that the initial condition \eqref{initial} holds. Then, there exists a maximal existence time \( T_{\rm max} \in (0, \infty] \) and a unique pair of functions \( u \) and \( v \) satisfying  
    \[
    u, v \in C^0 \left(\bar{\Omega} \times [0,T_{\rm max})\right) \cap C^{2,1} \left(\bar{\Omega} \times (0,T_{\rm max})\right),
    \]  
    which solve the system \eqref{1} with the boundary conditions \eqref{bdry} in the classical sense. Moreover, \( u > 0 \) and \( v \geq e^{-t} \inf_{x \in \Omega} v_0(x) \) in \( \bar{\Omega} \times (0,T_{\rm max}) \). If \( T_{\rm max} < \infty \), then  
    \begin{align} \label{local-exist-1}
        \limsup_{t \to T_{\rm max}} \left\{ \| u(\cdot,t) \|_{L^\infty(\Omega)} + \| v(\cdot,t) \|_{W^{1, \infty}(\Omega)} \right\} = \infty.
    \end{align}  
\end{lemma}

\begin{proof}
The proof relies on a standard fixed-point argument in Banach spaces. For a detailed proof, we refer the reader to \cite{Zhao+Zheng}[Lemma 2.2].
\end{proof}
From now on, we refer to \( (u,v) \) as the unique solution of system \eqref{1} in \( \Omega \times (0, T_{\rm max}) \), where \( T_{\rm max} \in (0, \infty] \), as established in the previous lemma. Next, we recall a standard $L^p$ regularity result for parabolic equations as established in \cite{DM}[Lemma 4.1].
\begin{lemma} \label{C52.Para-Reg}
Let $\Omega \subset \mathbb{R}^n$ with $n \geq 2$ be an open bounded domain with smooth boundary. Assume that $p\geq 1$ and $q \geq 1$ satisfy 
\begin{equation*}
    \begin{cases}
     q &< \frac{np}{n-p},  \qquad \text{when } p<n,\\
     q &< \infty, \qquad \text{when } p=n,\\
      q &= \infty, \qquad \text{when } p>n.\\
     \end{cases}
\end{equation*}
Assuming $V_0 \in W^{1,q}(\Omega)$ and $V$ is a classical solution to the following system
\begin{equation}\label{C52.parabolic-equation}
    \begin{cases}
     V_t = \Delta V  - a V + f &\text{in } \Omega \times (0,T), \\ 
\frac{\partial V}{\partial \nu} =  0 & \text{on }\partial \Omega \times (0,T),\\ 
 V(\cdot,0)=V_0   & \text{in } \Omega
    \end{cases}
\end{equation}
where $a>0$ and $T\in (0,\infty]$. If $f \in L^\infty \left ( (0,T);L^p(\Omega) \right ) $, then $V  \in L^\infty \left ( (0,T);W^{1,q}(\Omega) \right )$.
\end{lemma}
The following lemma serves a key inequality for comparing \( u \) and \( \Delta v \). For a detailed proof, we refer interested readers to \cite{WMS}[Lemma 2.3].
\begin{lemma} \label{l1}
    Assuming that $\Omega \subset \mathbb{R}^n$ with $n \geq 1$,  $ p \in (n, \infty)$, and $T\in (0,\infty]$. Then there exists $C= C(n,p,\Omega)>0$ such that for any $t \in (t_0, T)$ with $t_0:= \max \left \{ 1, \frac{T}{2}\right \}$, the following holds
\begin{align}
    \int_{t_0}^t  e^{\frac{ps}{2}} \int_\Omega |\Delta g|^p\,dx  \, ds \leq C \left ( \int_{t_0}^t e^{\frac{ps}{2}} \int_\Omega |f|^p \,dx \, ds + e^{\frac{pt_0}{2}} \left \| \Delta g(\cdot, t_0) \right \|^p_{L^p(\Omega)}\right ),
\end{align}
for any $f \in L^p \left (  \Omega \times (0,T) \right ) $ and $g$ is a classical solution to the following system with initial condition $g_0 \in W^{2, p}(\Omega)$:
\begin{equation}
    \begin{cases}
     g_t = \Delta g  -  g + f &\text{in } \Omega \times (0,T), \\ 
\frac{\partial g}{\partial \nu} =  0 & \text{on }\partial \Omega \times (0,T),\\ 
 g(\cdot,0)=g_0   & \text{in } \Omega.
    \end{cases}
\end{equation}
\end{lemma}
The next lemma establishes a key inequality that will be used in the proof of Lemma \ref{Lp}. Its proof follows directly from Proposition 1.3 in \cite{Kurt+Shen} and Proposition 3.1 in \cite{Kurt} for \( p \geq 3 \). Here, we provide an alternative proof and extend part of their results to any \( p > 1 \).

 \begin{lemma} \label{LK-1}
    Let $\Omega \subset \mathbb{R}^n$ with $n \geq 2$ be an open bounded domain with smooth boundary. For any $p >1$, there exist positive constants $C_1=C_1(p)$ and $C_2=C_2(p)$ such that the following inequality holds
    \begin{align*}
        \int_\Omega \frac{|\nabla w|^{2p}}{w^p} \leq C_1 \int_\Omega |\Delta w|^p +C_2\int_\Omega w^p, \qquad \text{for all } t \in (0,T_{\rm max}).
    \end{align*}
    for any positive function $w \in C^2(\bar{\Omega})$ satisfying $\frac{\partial w}{\partial \nu}=0$ on $\partial \Omega$.
\end{lemma}

\begin{proof}
    Integrating by parts implies that 
    \begin{align*}
       \int_\Omega \frac{|\nabla w|^{2p}}{w^p} &= - \int_\Omega \nabla \left ( \frac{|\nabla w|^{2p-2} \nabla w}{w^p} \right )w \notag \\
       &= - \int_\Omega |\nabla w|^{2(p-2)} \frac{\nabla |\nabla w|^2 \cdot \nabla w}{w^{p-1}} - \int_\Omega \frac{|\nabla w|^{2p-2} \Delta w}{w^{p-1}} + p \int_\Omega \frac{|\nabla w|^{2p}}{w^p}.
    \end{align*}
    Applying Young's inequality deduces that 
    \begin{align} \label{M.1}
         - \int_\Omega \frac{|\nabla w|^{2p-2} \Delta w}{w^{p-1}} \leq \frac{p-1}{2}\int_\Omega \frac{|\nabla w|^{2p}}{w^p} + c_1 \int_\Omega |\Delta w|^p,
    \end{align}
    where $c_1=c_1(p) >0$.  Therefore, we obtain that 
    \begin{align} \label{M.1'}
       \frac{p-1}{2} \int_\Omega \frac{|\nabla w|^{2p}}{w^p} &\leq  \int_\Omega |\nabla w|^{2(p-2)} \frac{\nabla |\nabla w|^2 \cdot \nabla w}{w^{p-1}} +c_1\int_\Omega |\Delta w|^p.
    \end{align}
    Thanks to the point-wise identity that 
    \begin{equation*}
        \nabla |\nabla w|^2 \cdot \nabla w = 2 \nabla w \cdot D^2 w \cdot \nabla w
    \end{equation*}
   and Young's inequality, we have 
    \begin{align} \label{M.2}
       \int_\Omega |\nabla w|^{2(p-2)} \frac{\nabla |\nabla w|^2 \cdot \nabla w}{w^{p-1}}  &= 2\int_\Omega |\nabla w|^{2(p-2)} \frac{\nabla w \cdot D^2 w \cdot \nabla w}{w^{p-1}} \notag \\
       &\leq 2 \int_\Omega \frac{|\nabla w|^{2(p-1)}}{w^{p-1}} |D^2 w| \notag \\
        &\leq \frac{p-1}{4} \int_\Omega \frac{|\nabla w|^{2p}}{w^p} + c_2 \int_\Omega |D^2 w|^p,
    \end{align}
    where $c_2=c_2(p)>0$.
 By elliptic regularity theory (see \cite{Elliptic}[Theorem 9.9] or \cite{zh}[Lemma 2.5]), it follows that
\begin{equation} \label{M.3}
     \int_\Omega |D^2 w|^p \leq c_3 \int_\Omega |\Delta w - w|^p \leq c_4 \int_\Omega |\Delta w|^p+ c_4 \int_\Omega w^p,
\end{equation}
for some $c_3>0$ and $c_4>0$.
From  \eqref{M.1'}, \eqref{M.2}, and \eqref{M.3}, we obtain that
 \begin{align*}
 \frac{p-1}{4}   \int_\Omega \frac{|\nabla w|^{2p}}{w^p}  &\leq \left ( c_2c_4 +c_1 \right ) \int_\Omega |\Delta w  |^{p} + c_2c_4 \int_\Omega w^p,
 \end{align*}
 which completes the proof.
\end{proof}

\section{A Priori Estimates}\label{s3}
In this section, we will establish some important estimates for solutions to prepare for the proof of the main result. Let us begin with an $L^1$ bound for $u$ as follows:
\begin{lemma} \label{L1-est}
    There exist $m>0$ such that
    \begin{align*}
        \int_\Omega u(\cdot,t) + \int_\Omega v(\cdot,t ) \leq m, \qquad \text{for all }t \in (0, T_{\rm max}),
    \end{align*}
    where $\tau = \min \left \{ \frac{T_{\rm max}}{2},1 \right \}$.
\end{lemma}
\begin{proof}
     Integrating the first equation of \eqref{1} over $\Omega$ and applying Holder's inequality yields that
    \begin{align*} 
        \frac{d}{dt} \int_\Omega  u(\cdot,t)&= r  \int_\Omega  u(\cdot,t) - \mu  \int_\Omega  u^2(\cdot,t) \notag \\
        &\leq r  \int_\Omega  u(\cdot,t) - \frac{\mu}{|\Omega|} \left ( \int_\Omega  u(\cdot,t) \right )^2,
    \end{align*}
    for all $t \in (0,T_{\rm max})$. From the comparison principle, it follows that
    \begin{align}\label{L1-est.1}
        \int_\Omega  u(\cdot,t) \leq c_1:= \max \left \{ \frac{r|\Omega|}{\mu}, \int_\Omega u_0. \right \},
    \end{align}
     for all $t \in (0,T_{\rm max})$. Integrating the second equation of \eqref{1} over $\Omega$ and applying Lemma \ref{L1-est} yields
    \begin{align*}
        \frac{d}{dt}\int_\Omega v + \alpha \int_\Omega v &= \beta \int_\Omega u \leq \beta c_1.
    \end{align*}
   =Therefore, applying Gronwall's inequality to this implies that 
    \begin{align}\label{L1-est.2}
        \sup_{t\in (0,T_{\rm max})} \int_\Omega v(\cdot,t) \leq \max \left \{ \int_\Omega v_0, \frac{\beta c_1}{\alpha} \right \}.
    \end{align}
    From \eqref{L1-est.1} and \eqref{L1-est.2}, the proof is now complete.
\end{proof}

\begin{lemma} \label{v-Lp}
    For any $\lambda>0$ there exists $C=C(\lambda)>0$ such that 
    \begin{align}
        \frac{d}{dt}\int_\Omega v^p + \lambda \int_\Omega v^p \leq \beta \int_\Omega u^p+C,
    \end{align}
    for any $p>1$.
\end{lemma}
\begin{proof}
    Multiplying the second equation of \eqref{1} by $pv^{p-1}$ and integrating by parts yields
    \begin{align*}
        \frac{d}{dt} \int_\Omega v^p &= -\frac{4(p-1)}{p} \int_\Omega |\nabla v^{\frac{p}{2}}|^2 - \alpha p \int_\Omega v^p +\beta p \int_\Omega u v^{p-1}.
    \end{align*}
   This, together with Young's inequality implies that 
   \begin{align} \label{v-Lp.1}
         \frac{d}{dt} \int_\Omega v^p + \lambda \int_\Omega v^p \leq  -\frac{4(p-1)}{p} \int_\Omega |\nabla v^{\frac{p}{2}}|^2 +c_1 \int_\Omega v^p +\beta \int_\Omega u^p
   \end{align}
   where $c_1= (\lambda -\alpha p)_+ +\beta(p-1) $. Applying Gagliardo–Nirenberg interpolation inequality, Young's inequality and Lemma \ref{L1-est} deduces that 
   \begin{align*}
       c_1 \int_\Omega v^p &\leq c_2 \left ( \int_\Omega |\nabla v^\frac{p}{2}|^2 \right )^{\frac{np-n}{np-n+2}} \left ( \int_\Omega v \right )^{\frac{2p}{np+2-n}} +c_2 \left ( \int_\Omega v \right )^p \notag \\
       &\leq  -\frac{4(p-1)}{p} \int_\Omega |\nabla v^{\frac{p}{2}}|^2  +c_3,
   \end{align*}
   for some $c_2>0$ and $c_3>0$. This, together with \eqref{v-Lp.1} proves the desired result.
\end{proof}

We can now establish the key estimate for solution of \eqref{1} as in the following lemma.
 \begin{lemma} \label{Lp}
     Let $p >2$ and $1\leq q<p-1$. There exists $\mu_0=\mu_0(p,q,n,k,\chi,\beta)>0$ such that for any $\mu>\mu_0$, the following holds
     \begin{align*}
         \sup_{t \in (0,T_{\rm max})} \left \{ \int_\Omega u^p +u^pv^{-q} +\int_\Omega v^{p+1} \right \} < \infty.
     \end{align*}
 \end{lemma}
    \begin{proof}
        By integrating by parts and using \eqref{1}, we obtain that
    \begin{align} \label{Lp.1}
        \frac{d}{dt} \int_\Omega u^p v^{-q} &= p \int_\Omega u^{p-1}v^{-q}u_t -q \int_\Omega u^p v^{-q-1}v_t \notag \\
        &= p \int_\Omega u^{p-1}v^{-q} \left ( \Delta u -\chi \nabla \cdot \left ( uv^{-k} \nabla v \right ) +ru -\mu u^2 \right ) \notag \\
        &\quad -q \int_\Omega u^p v^{-q-1} \left ( \Delta v- \alpha v +\beta u \right ) \notag \\
        &=- p(p-1)\int_\Omega u^{p-2}v^{-q}|\nabla u|^2 +2pq \int_\Omega u^{p-1}v^{-q-1} \nabla u \cdot \nabla v  \notag\\
        &\quad + p(p-1)\chi \int_\Omega u^{p-1}v^{-q-k} \nabla u \cdot \nabla v -pq\chi \int_\Omega u^pv^{-q-k-1}|\nabla v|^2   + (rp+q\alpha) \int_\Omega u^p v^{-q}\notag \\
        &\quad -\mu p \int_\Omega u^{p+1}v^{-q} -q(q+1)\int_\Omega u^p v^{-q-2}|\nabla v|^2- q\beta \int_\Omega u^{p+1}v^{-q-1},
        \end{align}
        and 
        \begin{align}\label{Lp.2}
            \frac{d}{dt} \int_\Omega u^p &= -\frac{4(p-1)}{p}\int_\Omega |\nabla u^\frac{p}{2}|^2 + \chi k (p-1)\int_\Omega u^p v^{-1-k}|\nabla v|^2 -\chi(p-1)\int_\Omega u^pv^{-k} \Delta v \notag \\
            &\quad + rp \int_\Omega u^p -\mu p \int_\Omega u^{p+1}.
        \end{align}
    Employing Young's inequalities yields
    \begin{align}\label{Lp.3}
         2pq \int_\Omega u^{p-1}v^{-q-1} \nabla u \cdot \nabla v \leq (p(p-1)-\varepsilon_1) \int_\Omega u^{p-2}v^{-q} |\nabla u|^2+ \frac{p^2q^2}{p(p-1)-\varepsilon_1} \int_\Omega u^p v^{-q-2}|\nabla v|^2,
    \end{align}
    and 
    \begin{align}\label{Lp.4}
        p(p-1)\chi \int_\Omega u^{p-1}v^{-q-k} \nabla u \cdot \nabla v &\leq \varepsilon_1 \int_\Omega u^{p-2} v^{-q}|\nabla u|^2 + \frac{p^2(p-1)^2\chi^2}{4\varepsilon_1}\int_\Omega u^p v^{-q-2k}|\nabla v|^2 \notag \\
        &\leq \varepsilon_1 \int_\Omega u^{p-2} v^{-q}|\nabla u|^2 + \varepsilon_2 \int_\Omega u^p v^{-q-2}|\nabla v|^2 +c_1 \int_\Omega u^pv^{-1}|\nabla v|^2,
    \end{align}
    where $c_1=c_1(\epsilon_1,\epsilon_2, \chi,k,p)>0$ and the last inequality holds due to \(q+2> q+2k>1\). Similarly, we also have
    \begin{align}\label{Lp.5}
        \chi k (p-1)\int_\Omega u^p v^{-1-k}|\nabla v|^2 \leq \epsilon_2 \int_\Omega u^p v^{-q-2}|\nabla v|^2 +c_2 \int_\Omega u^p v^{-1}|\nabla v|^2
    \end{align}
     where $c_2=c_2(\epsilon_2, \chi,k,p)>0$, and 
     \begin{align}\label{Lp.6}
         -\chi(p-1)\int_\Omega u^pv^{-k} \Delta v &\leq  \chi(p-1) \int_\Omega \frac{u^{p+1}}{v^{\frac{(p+1)k}{p}}} + \chi \int_\Omega |\Delta v|^{p+1} \notag\\
         &\leq \frac{q\beta }{2} \int_\Omega u^{p+1}v^{-q-1} + c_3 \int_\Omega u^{p+1} + \chi \int_\Omega |\Delta v|^{p+1}, 
     \end{align}
     where $c_3=c_3(p,q,\beta,\chi,k)$ and the last inequality holds since $0<\frac{(p+1)k}{p}< q+1$. Using Young's inequality again entails that 
    \begin{align}\label{Lp.7}
        (\frac{p+1}{2}+rp+q\alpha) \int_\Omega u^p v^{-q} &\leq \frac{\beta q}{2} \int_\Omega u^{p+1}v^{-q-1}+c_4 \int_\Omega v^{p-q} \notag \\
         &\leq \frac{\beta q}{2} \int_\Omega u^{p+1}v^{-q-1}+\int_\Omega v^{p+1}+c_5.
    \end{align}
    for some $c_4>0$ and $c_5>0$. It is akin to the above estimates that
    \begin{align}\label{Lp.8}
        (rp+\frac{p+1}{2}) \int_\Omega u^p \leq \frac{\mu p}{2} \int_\Omega u^{p+1} + c_6,
    \end{align}
    where $c_6>0$. Setting $y(t)= \int_\Omega u^p + \int_\Omega u^p v^{-q}$ and collecting from \eqref{Lp.1} to \eqref{Lp.8} entails that 
    \begin{align*}
        y'(t)+\frac{p+1}{2}y(t) &\leq \left ( 2 \epsilon_2 + \frac{p^2q^2}{p(p-1)-\epsilon_1} -q(q+1) \right ) \int_\Omega u^p v^{-q-2}|\nabla v|^2 +c_7 \int_\Omega u^p v^{-1}|\nabla v|^2 \notag \\
        &\quad+\left ( c_3-\frac{\mu p}{2} \right )\int_\Omega u^{p+1} + \int_\Omega v^{p+1}+ \chi \int_\Omega |\Delta v|^{p+1} +c_5+c_6,
    \end{align*}
     where $c_7=c_1+c_2$ independent of $\mu$. Choosing $\epsilon_1 = \frac{p(p-q-1)}{2(q+1)}$ and $\epsilon_2 = \frac{q(q+1)}{2}- \frac{p^2q^2}{2(p(p-1)- \epsilon_1)}$ leads to 
   \[
    2 \epsilon_2 + \frac{p^2q^2}{p(p-1)-\epsilon_1} -q(q+1) \leq 0,
   \]
   due to $q+1<p$. Therefore, we obtain that
   \begin{align}\label{Lp.9}
           y'(t)+\frac{p+1}{2}y(t) \leq c_7 \int_\Omega u^p v^{-1}|\nabla v|^2 +\left ( c_3-\frac{\mu p}{2} \right )\int_\Omega u^{p+1} + \int_\Omega v^{p+1}+ \chi \int_\Omega |\Delta v|^{p+1}+c_5+c_6.
   \end{align}
   Applying Young's inequality and Lemma \ref{LK-1} implies that 
   \begin{align}\label{Lp.9'}
        c_7 \int_\Omega u^p v^{-1}|\nabla v|^2 &\leq c_7 \int_\Omega u^{p+1} + c_7 \int_\Omega \frac{|\nabla v|^{2p+2}}{v^{p+1}} \notag \\
        &\leq c_7 \int_\Omega u^{p+1} +c_8 \int_\Omega |\Delta v|^{p+1} +c_9 \int_\Omega v^{p+1},
   \end{align}
   where $c_8=c_8(p,q,k,n,\chi)>0$ and $c_9>0$. Using Lemma \ref{v-Lp} with $\lambda=  \frac{p+1}{2} +c_9+1$ entails that 
   \begin{align} \label{Lp.9''}
    \frac{d}{dt} \int_\Omega v^{p+1} + \left ( \frac{p+1}{2} +c_9 +1\right ) \int_\Omega v^{p+1} \leq \beta  \int_\Omega u^{p+1}  +c_{10}
   \end{align}
   where $c_{10}>0$. Denoting $h(t)= y(t) +\int_\Omega v^{p+1}(\cdot,t)$ and combining \eqref{Lp.9}, \eqref{Lp.9'}, and \eqref{Lp.9''} deduces that 
   \begin{align}
       h'(t)+\frac{p+1}{2} h(t) &\leq \left ( c_3+c_7+\beta -\frac{\mu p}{2} \right )\int_\Omega u^{p+1} + c_9 \int_\Omega |\Delta v|^{p+1}+c_{11},
   \end{align}
   where $c_9= c_8+\chi$ and $c_{11}=c_5+c_6+c_{10}$. Multiplying this by $e^{\frac{p+1}{2}t}$, integrating over $(t_0,t)$ where $t_0 = \min \left\{ 1, \frac{T_{\rm max}}{2}\right \}$ and applying Lemma \ref{l1} implies that
   \begin{align*}
      e^{\frac{p+1}{2}t} h(t )  - e^{\frac{p+1}{2}t_0} h(t_0 )&\leq \left ( c_3+c_7+\beta -\frac{\mu p}{2} \right ) \int_{t_0}^t e^{\frac{p+1}{2}s} \int_\Omega u^{p+1}(\cdot,s)\, ds +c_{9} \int_{t_0}^t e^{\frac{p+1}{2}s} \int_\Omega |\Delta v(\cdot,s)|^{p+1} \notag \\
      &\quad + \frac{2c_{11}}{p+1} \left ( e^{\frac{p+1}{2}t}-  e^{\frac{p+1}{2}t_0} \right )   \notag \\
      &\leq \left (c_{12} - \frac{\mu p}{2} \right )\int_{t_0}^t e^{\frac{p+1}{2}s} \int_\Omega u^{p+1}(\cdot,s)\, ds +c_{13} e^{\frac{p+1}{2}t_0} \int_\Omega |\Delta v(\cdot,t_0)|^{p+1} \notag \\
       &\quad + \frac{2c_{11}}{p+1} \left ( e^{\frac{p+1}{2}t}-  e^{\frac{p+1}{2}t_0} \right )  
   \end{align*}
   where $c_{12}=c_{12}(p,q,k,n,\chi, \beta )>0$ and $c_{13}>0$.
  By choosing $\mu_0=  \frac{2c_{12}}{p}$, then for any $\mu >\mu_0$ we obtain
   \begin{align}\label{Lp.10}
       h(t) \leq h(t_0) +c_{13} e^{\frac{p+1}{2}t_0} \int_\Omega |\Delta v(\cdot,t_0)|^{p+1} + \frac{2c_{11}}{p+1} = c_{14}
   \end{align}
   for all $t \in (t_0, T_{\rm max})$. Since $u$, $v$ and $\frac{1}{v}$ is bounded in $\Omega \times (0,t_0)$ due to Lemma \ref{local-exist}, we have $\sup_{t\in (0,t_0]}h(t) =c_{15}< \infty$, which in conjunction with \eqref{Lp.10} entails that $\sup_{t \in (0,T_{\rm max})}h(t) \leq \max \left \{ c_{14},c_{15} \right \}$. The proof is now complete.
    \end{proof}

    As a direct consequence of the above lemma, we can now establish the following estimate for solutions, which will be used in the proof of Theorem \ref{thm1}.
    \begin{lemma} \label{uv}
        For any $p> \max \left \{ n, \frac{1}{1-k}, \frac{1}{k} \right \}$ there exists $\mu^*=\mu^*(p,n,k,\chi, \beta )>0$ such that for any $\mu>\mu^*$ the following holds 
        \begin{align*}
            \sup_{t \in (0,T_{\rm max})}\int_\Omega u^p(\cdot,t) \frac{|\nabla v(\cdot,t)|^p}{v^{kp}(\cdot,t)} <\infty. 
        \end{align*}
    \end{lemma}
    \begin{proof}
       The condition $p> \max \left \{ n, \frac{1}{1-k}, \frac{1}{k} \right \}$ entails that $1\leq kp<p-1$. From Lemma \ref{Lp}, there exists $\mu^* = \mu^*(p,n,k,\chi, \beta )>0$ such that for any $\mu> \mu^*$, we obtain 
       \begin{align} \label{uv.1}
           \int_\Omega u^p + \int_\Omega \frac{u^p}{v^{kp}} \leq c_1, \quad \text{for all }t \in (0,T_{\rm max}),
       \end{align}
       where $c_1=c_1(p)>0$. Since $u \in L^\infty \left ( (0,T_{\rm max});L^p(\Omega) \right )$ for some $p>n$, Lemma \ref{C52.Para-Reg} asserts that 
       \begin{align} \label{uv.2}
           \left \| \nabla v(\cdot,t) \right \|_{L^\infty(\Omega)} \leq c_2, \quad \text{for all } t\in (0,T_{\rm max}),
       \end{align}
       where $c_2>0$. From \eqref{uv.1} and \eqref{uv.2}, we obtain that 
       \begin{align*}
           \int_\Omega u^p \frac{|\nabla v|^p}{v^{kp}} \leq c_2^pc_1, \quad \text{for all } t\in (0,T_{\rm max}),
       \end{align*}
       which completes the proof.
    \end{proof}
    \section{Proof of the main result} \label{s4}
  By applying standard $L^p$-$L^q$ estimates for the Neumann heat semigroup and following the argument in \cite{Winkler-logistic} [Proof of Theorem 0.1], Theorem \ref{thm1} is immediately proven. For the readers' convenience, we provide a detailed proof here.
    \begin{proof}[Proof of Theorem \ref{thm1}]
    We have 
    \begin{align} \label{pr.1}
         u(\cdot,t)=  e^{(t-t_0) \Delta }u(\cdot,t_0) + \int_{t_0} ^t e^{(t-s)\Delta} \nabla \cdot \left ( u\frac{\nabla v}{v^k} \right )(\cdot,s)\, ds +\int_{t_0}^t e^{(t-s)\Delta } (ru -\mu u^2) \, ds, 
    \end{align}
    where $t \in (0,T_{\rm max})$ and $t_0 = \max \left \{ t-1,0 \right \} $. Noting that $ru -\mu u^2 \leq c_1$, where $c_1= \frac{r^2}{4\mu}$, we obtain
        \begin{align} \label{pr.2}
            \int_{t_0}^t e^{(t-s)\Delta } (ru -\mu u^2) \, ds \leq c_1(t-t_0).
        \end{align}
    One can verify that 
    \begin{align} \label{pr.3}
        e^{(t-t_0) \Delta }u(\cdot,t_0) \leq \left \| e^{(t-t_0) \Delta }u(\cdot,t_0) \right \|_{L^\infty(\Omega) } \leq c_2 (1+(t-t_0)^{-\frac{n}{2}}),
    \end{align}
    for some $c_2>0$. Lemma \ref{uv} asserts that for $q= \max \left \{ n,\frac{1}{k}, \frac{1}{1-k} \right \}+1$, there exists $\mu_0=\mu_0(n,k,r,\chi,\beta)>0$ such that for any $\mu > \mu_0$, we obtain
    \begin{align*}
        \sup_{t\in (0,T_{\rm max})}\left \| u(\cdot,t) \frac{\nabla v(\cdot,t) }{v^{k}(\cdot,t)} \right \|_{L^q(\Omega)} <\infty.
    \end{align*}
    This, together with standard $L^p-L^q$ estimates for $(e^{t \Delta})_{t\geq 0}$ (see  \cite{Winkler-2010}[Lemma 1.3]) deduces that
    \begin{align} \label{pr.4}
        \int_{t_0} ^t e^{(t-s)\Delta} \nabla \cdot \left ( u\frac{\nabla v}{v^k} \right )(\cdot,s)\, ds  &\leq  \int_{t_0} ^t \left \| e^{(t-s)\Delta} \nabla \cdot \left ( u\frac{\nabla v}{v^k} \right )(\cdot,s) \right \|_{L^\infty(\Omega)} \, ds \notag \\
        &\leq  c_3\int_{t_0} ^t (t-s)^{-\frac{n}{2p}} \left \|  e^{\frac{t-s}{2} \Delta } \nabla \cdot \left (u(\cdot,s) \frac{\nabla v(\cdot,s) }{v^{k}(\cdot,s)} \right ) \right \|_{L^p(\Omega)} \, ds \notag\\
        &\leq c_4 \int_{t_0}^t (t-s)^{-\frac{n}{2p}} \cdot (t-s)^{-\frac{1}{2}-\frac{n}{2} \left ( \frac{1}{q}-\frac{1}{p} \right )} \left \| u(\cdot,s) \frac{\nabla v(\cdot,s) }{v^{k}(\cdot,s)} \right \|_{L^q(\Omega)} \, ds \notag \\
        &\leq c_5  \int_{t_0}^t (t-s)^{ -\frac{1}{2}-\frac{n}{2q}}  \, ds \notag \\
        &\leq c_5 (t-t_0)^{\frac{1}{2}-\frac{n}{2q}}
    \end{align}
    where $p>q$, and $c_3$, $c_4$, $c_5$ are positive constants. Combining \eqref{pr.1}, \eqref{pr.2}, \eqref{pr.3}, and \eqref{pr.4} yields
    \begin{align}
        \left \|u(\cdot,t) \right \|_{L^\infty(\Omega)} &\leq  c_1(t -t_0 ) + c_2 (1+(t-t_0)^{-\frac{n}{2}}) +c_5 (t-t_0)^{\frac{1}{2}-\frac{n}{2q}} \notag \\
        &\leq c_6 (t^{-\frac{n}{2}}+1), \quad \text{for all }t \in (0,T_{\rm max}),
    \end{align}
     which entails that  $u$ is bounded in $\Omega \times (0,T_{\rm max})$ since $u$ is bounded in $\Omega \times (0, \min \left \{1, \frac{T_{\rm max}}{2} \right \})$ due to Lemma \ref{local-exist}. By Lemma \ref{C52.Para-Reg}, it follows that $v \in L^\infty \left ((0,T_{\rm max}); W^{1,\infty}(\Omega) \right )$. Therefore, the extensibility property of the solutions as established in Lemma \eqref{local-exist-1} implies that $T_{\rm max} = \infty$ and that 
    \begin{align*}
        \sup_{t> 0} \left \{ \left \| u(\cdot,t) \right \|_{L^\infty(\Omega)} + \left \| v(\cdot,t) \right \|_{W^{1, \infty}(\Omega)} \right \} <\infty.
    \end{align*}
     
    \end{proof}


\begin{thebibliography}{99}

\bibitem{MKTAM}
M. Aida, K. Osaki, T. Tsujikawa, A. Yagi, and M. Mimura.
\newblock Chemotaxis and growth system with singular sensitivity function.
\newblock \emph{Nonlinear Analysis: Real World Applications}, 06:323--336, 2005.

\bibitem{Dolbeault1}
A. Blanchet, J. Dolbeault, and B. Perthame.
\newblock Two-dimensional Keller-Segel model: Optimal critical mass and qualitative properties of the solutions.
\newblock \emph{Electronic Journal of Differential Equations}, May 2006.

\bibitem{Dolbeault}
J. Dolbeault and B. Perthame.
\newblock A Optimal critical mass in the two dimensional Keller–Segel model in \(\mathbb{R}^2\).
\newblock \emph{C. R. Acad. Sci. Paris, Ser. I}, 339:611--616, 2004.

\bibitem{DM}
D. Horstmann and M. Winkler
\newblock Boundedness vs. blow-up in a chemotaxis system.
\newblock \emph{J. Differential Equations}, 215: 52-107, 2005.

\bibitem{Fujie}
K. Fujie.
\newblock Boundedness in a fully parabolic chemotaxis system with singular sensitivity.
\newblock \emph{Journal of Mathematical Analysis and Applications}, 424:675--684, 2015.

\bibitem{FWY}
K. Fujie, M. Winkler, and T. Yokota.
\newblock Blow-up prevention by logistic sources in a parabolic–elliptic Keller–Segel system with singular sensitivity.
\newblock \emph{Nonlinear Analysis: Theory, Methods and Applications}, 109:56--71, 2014.
\newblock \url{https://doi.org/10.1016/j.na.2014.06.017}

\bibitem{Elliptic}
D. Gilbarg and N. S. Trudinger.
\newblock \emph{Elliptic Partial Differential Equations of Second Order}.
\newblock Springer, 2001.
\newblock ISBN: 978-3-642-61798-0.

\bibitem{Keller2}
E. F. Keller and L. A. Segel.
\newblock A model for chemotaxis.
\newblock \emph{Journal of Theoretical Biology}, 30:225--234, 1971.

\bibitem{Keller}
E. F. Keller and L. A. Segel.
\newblock Initiation of slime mold aggregation viewed as an instability.
\newblock \emph{Journal of Theoretical Biology}, 26:399--415, 1970.

\bibitem{Keller+Segel}
E. F. Keller and L. A. Segel.
\newblock Traveling bands of chemotactic bacteria: a theoretical analysis.
\newblock \emph{Journal of Theoretical Biology}, 30:235--248, 1971.

\bibitem{Kurt+Shen}
H. Kurt and W. Shen.
\newblock Finite-time blow-up prevention by logistic source in parabolic-elliptic chemotaxis models with singular sensitivity in any dimensional setting.
\newblock \emph{SIAM on Mathematical Analysis}, 53:973--1003, 2021.

\bibitem{Kurt}
H.I. Kurt.
\newblock Boundedness in a chemotaxis system with weak singular sensitivity and logistic kinetics in any dimensional setting.
\newblock \emph{Journal of Differential Equations}, 416:1429--1461, 2025.

\bibitem{Lankeit_2016}
J. Lankeit.
\newblock A new approach toward boundedness in a two–dimensional parabolic chemotaxis system with singular sensitivity.
\newblock \emph{Math. Methods Appl. Sci.}, 39:394--404, 2016.

\bibitem{minh7}
M. Le.
\newblock Absence of blow-up in a fully parabolic chemotaxis system with weak singular sensitivity and logistic damping in dimension two.
\newblock \emph{Preprint}, 2025.

\bibitem{Minh6}
M. Le.
\newblock Boundedness in a chemotaxis system with weakly singular sensitivity in dimension two with arbitrary sub-quadratic degradation sources.
\newblock \emph{J. Math. Anal. Appl.}, 542:128803, 2025.
\newblock \texttt{doi:10.1016/j.jmaa.2024.128803}

\bibitem{Nagai1}
T. Nagai.
\newblock Blow-up of radially symmetric solutions to a chemotaxis system.
\newblock \emph{Adv. Math. Sci. Appl.}, 5:581--601, 1995.

\bibitem{Nagai2}
T. Nagai.
\newblock Global existence and blowup of solutions to a chemotaxis system.
\newblock \emph{Nonlinear Analysis}, 47:777--787, 2001.

\bibitem{NSY}
T. Nagai, T. Senba, and K. Yoshida.
\newblock Application of the Trudinger-Moser inequality to a Parabolic System of Chemotaxis.
\newblock \emph{Funkcilaj Ekvacioj}, 40:411--433, 1997.

\bibitem{Nagai4}
T. Nagai, T. Senba, and T. Suzuki.
\newblock Blowup of nonradial solutions to parabolic-elliptic systems modeling chemotaxis in two-dimensional domains.
\newblock \emph{J. Inequal. Appl.}, 6:37--55, 2001.

\bibitem{Nagai3}
T. Nagai, T. Senba, and T. Suzuki.
\newblock Chemotaxis collapse in a parabolic system of mathematical biology.
\newblock \emph{Hiroshima Math. J.}, 20:463--497, 2000.

\bibitem{Tello+Winkler}
J. Tello and M. Winkler.
\newblock A chemotaxis system with logistic source.
\newblock \emph{Comm. Partial Differential Equations}, 32:849--877, 2007.

\bibitem{WMS}
S. Zheng Wang and M. Zhuang.
\newblock Positive effects of repulsion in boundedness in a fully parabolic attraction–repulsion chemotaxis system with logistic source.
\newblock \emph{Journal Differential Equations}, 264:2011--2027, 2018.

\bibitem{Winkler-2010}
M. Winkler.
\newblock Aggregation vs. global diffusive behavior in the higher-dimensional Keller–Segel model.
\newblock \emph{J. Differential Equations}, 248(12):2889--2905, 2010.

\bibitem{Winkler-logistic}
M. Winkler.
\newblock Boundedness in the Higher-Dimensional Parabolic-Parabolic Chemotaxis System with Logistic Source.
\newblock \emph{Communications in Partial Differential Equations}, 35:1516--1537, 2010.

\bibitem{Winkler+2011}
M. Winkler.
\newblock Global solutions in a fully parabolic chemotaxis system with singular sensitivity.
\newblock \emph{Mathematical Methods in the Applied Sciences}, 34:176--190, 2011.

\bibitem{zh}
X. Zhao.
\newblock Boundedness in a logistic chemotaxis system with weakly singular sensitivity in dimension two.
\newblock \emph{Nonlinearity}, 36:3909--3938, 2023.

\bibitem{Zhao+Zheng}
X. Zhao and S. Zheng.
\newblock Global boundedness to a chemotaxis system with singular sensitivity and logistic source.
\newblock \emph{Z. Angew. Math. Phys.}, 68, 2017.



\end{thebibliography}
    \end{document}